\documentclass[11pt,a4paper]{article}
\usepackage{amsmath,amsfonts,amssymb}
\usepackage{graphicx}

\newtheorem{theorem}{Theorem}[section]
\newtheorem{lemma}[theorem]{Lemma}
\newtheorem{corollary}[theorem]{Corollary}
\newtheorem{proposition}[theorem]{Proposition}
\newtheorem{remark}[theorem]{Remark}

\newcommand{\beq}{\begin{equation}}
\newcommand{\eeq}{\end{equation}}
\newcommand{\beqa}{\begin{eqnarray}}
\newcommand{\eeqa}{\end{eqnarray}}
\newcommand{\beqas}{\begin{eqnarray*}}
\newcommand{\eeqas}{\end{eqnarray*}}
\newcommand{\bi}{\begin{itemize}}
\newcommand{\ei}{\end{itemize}}

\def\QED{\ifhmode\unskip\nobreak\fi\ifmmode\ifinner\else\hskip5pt\fi\fi
  \hbox{\hskip5pt\vrule width5pt height5pt depth1.5pt\hskip1pt}}

\def\tx{{\tilde x}}

\def\tA{\widetilde A}

\def\ts{{\tilde s}}

\def\RR{{\mathbb{R}}}

\title{\Large An Extension of Chubanov's Polynomial-Time Linear Programming Algorithm to Second-Order Cone Programming}

\author{
Tomonari Kitahara %
\thanks{Tokyo Institute of Technology
  (Email: {\tt kitahara.t.ab@m.titech.ac.jp}).}
\and
Takashi Tsuchiya%
\thanks{National Graduate Institute for Policy Studies
(Email: {\tt tsuchiya@grips.ac.jp}). 
}
}
\date{November 2016\\
 (Revised: December 2016, January 2017)}
\begin{document}

\maketitle

\begin{abstract}

Recently, Chubanov proposed an interesting new polynomial-time algorithm for linear program.
In this paper, we extend his algorithm to second-order cone programming.

\vskip14pt

\end{abstract}

\section{Introduction}

In linear programming, 
the ellipsoid method \cite{Kha1,Kha2} and the interior-point method \cite{Kar,NeNe} were 
the only two algorithms which enjoy polynomiality for a long time.
Recently, an interesting new polynomial-time algorithm was proposed by Chubanov \cite{Ch10,Ch12,Ch13}.
%%TT161130
Related studies include, for instance, \cite{BaDeJu,LiRoTe,Pe16,Ro15}. 
In this paper, we develop a word-by-word extension of Chubanov's algorithm to second-order cone 
programming \cite{AlGo,Cu15,MoTs,NeNe,NeTo,Tsu}. Among the related works, Pe\~na and Soheili \cite{Pe16}
developed a polynomial-time projection and rescaling algorithm for a symmetric cone feasibility problem. 
Their algorithm utilizes Chubanov's idea and is closely related to ours in its direction.
We briefly compare the two approaches later to highlight the difference.
%%TT161130end

%%TT161130
%Chubanov's algorithm is to find a point in the intersection of a linear space and a unit hypercube,
%i.e., the direct product of 0-1 segment.
%The algorithm finds an interior feasible solution or detect a variable whose value cannot be greater than 1/2 
%in the feasible region and then that coordinate is rescaled by one.  
%Similarly, we find a feasible solution to the system of
%
%As a remarkable result in this direction, we mention a very recent work \cite{Pe16} by 
%Pe\~na and Soheili \cite{Pe16} in which they developed an interesting projection and rescaling 
%algorithm for a symmetric cone feasibility problem built upon their rescaled perceptron algorithm and 
%Chubanov's algorithm.  
%Our extension is different from theirs even when restricted to second-order cone programming. 
%While the algorithm in \cite{Pe16} is based on the idea of rescaling the system to improve its condition number,
%our algorithm finds a cut and generates a series of shrinking convex bodies of the same type enclosing the feasible 
%region.  We would say that our algorithm is a direct and a word-by-word extension of Chubanov's algorithm.
%This is nice feature which our extension enjoys.
%Such connection is not so apparent in P\~na and Soheili's algorithm.

%leading to another generalization of Chubanov's idea.  
%%TT161130end

Consider the following homogeneous second-order cone programming feasibility problem (P):
\[
{\rm(P)}\ \ \ 
\hbox{find}_{(x_1;\ldots;x_n)}\ 
\sum A_i x_i = 0,\ \ x_i \in {\cal K}_i,\ i=1, \ldots, n,
\]
where $A_i\in \RR^{m\times {d_i}}$ for each $i$ and ${\cal K}_i \subseteq \RR^{d_i}$ is either a half-line or a second-order cone, i.e.,
\begin{eqnarray*}
{\cal K}_i &=& \{x_i\in \RR|\ x_i \geq 0\} \ \hbox{(if\ $d_i = 1$)},\\
{\cal K}_i &=& \{x_i = (x_{i0};x_{i1})\in \RR\times \RR^{{d_i}-1}|\ \|x_{i1}\|\leq x_{i0} \} \ \hbox{(if\ $d_i \geq 2$)}.
\end{eqnarray*}
We assume that vectors are in column form by default and the vertical concatenation of two vectors $a$ and $b$ is written as $(a;b)$.
We denote by $\cal SOC$ and $\cal LI$ the set of indices $i$ where ${\cal K}_i$ is a second-order cone and 
a half-line, respectively.
The dual problem (D) is
\[
{\rm(D)}\ \ \ 
\hbox{find}_{(s_1;\ldots;s_n)}\ s_i=-\sum A_i ^T u,\ \ s_i \in {\cal K}_i,\ i=1, \ldots, n,\ \ u\in \RR^m.
\]

Throughout this paper we use a notation analogous to $(x_{i0};x_{i1}) \in {\cal K}_i$ concerning a second-order cone.  When we deal with
a vector in a space where a second-order cone is defined, 
the first element $x_{i0}$ with ``the index 0" always represents the center axis of a second-order cone, and the second element $x_{i1}$ with the
``index 1" represents the rotational part, unless otherwise stated.
In the following, for a cone ${\widetilde {\cal K}}$, say, we use the notations $x\succeq y$ and $x\succ y$ to mean that $x - y \in \widetilde{\cal K}$ and
$x-y \in {\rm int}(\widetilde {\cal K})$, respectively.
Letting $A = (A_1, \ldots, A_n)$, ${\cal K}={\cal K}_1\times \ldots \times {\cal K}_n$, (P) and (D) are written as
\[
{\rm (P)}\ {\rm find}_x\  
A x = 0,\ \ x \succeq 0,\ \ \ {\rm (D)}\ {\rm find}_{(s,u)} \ s = -A^T u,\ \ s \succeq 0,
\]
where $x=(x_1; \ldots; x_n) \in \RR^{\bar n}$,  $s=(s_1; \ldots; s_n) \in \RR^{\bar n}$, $y\in\RR^m$ and ${\bar n}=\sum_{i=1}^n d_i$.
For simplicity, we assume that $A \in \RR^{m\times {\bar n}}$ is row independent.

By generalized Gordan's theorem, (P) has an interior feasible solution if and only if
(D) does not have a nonzero solution (i.e., zero is the only solution to (D)), and 
(D) has an interior feasible solution if and only if
(P) does not have a nonzero solution.  If we let
\[
{\rm (GP)}\ {\rm find}_x\  
A x = 0,\ \ x \succ 0,\ \ \ {\rm (GD)}\ {\rm find}_{s} \ s = -A^T u,\ \ s \succeq 0, \ s\not=0,\ u\in \RR^m,
\]
the Generalized Gordan's Theorem says that (GP) has a solution if and only if (GD) does not have a solution
and (GD) has a solution if and only if (GP) does not have a solution.

Given a matrix $B$, say, let $P_B$ be an orthogonal projection matrix to ${\rm Ker}(B)$.  If $B$ is a row independent
matrix, $P_B =I -B^T (B B^T)^{-1}B$. The problem (P) is written as
\[
{\rm find}_x\ x = P_{A} y,\ \ x \succeq 0,\ y\in \RR^{\bar n}.
\]
and the problem (D) is written as
\[
{\rm find}_y\ 
P_{A} y=0,\ \  y \succeq 0,
\]
(where ``the free variable" $u$ is eliminated).

We will develop a polynomial-time algorithm for finding a solution to (GP) or (GD) or detecting
no $\varepsilon$-interior feasible solution exists to (P) (the definition of 
$\varepsilon$-interior feasible solution is given below). 
In Appendix we describe how we can solve approximately 
a general second-order cone program with a primal-dual interior feasible solution by the algorithm
developed in this paper.

The problem (P) is equivalent to finding an interior feasible solution to the following system.
%%TT161130
\[
Ax =0,\ \ \|x\|_\infty \leq 1,\ x \succeq 0.
\]
%%TT161130end
We denote by ${\cal F}$ the set of feasible solutions to this system.
We define the projection ${{\cal F}}_i$ of ${\cal F}$ onto the block $i$ as follows:
\begin{equation} \label{Fi}
{\cal F}_i = \{ x_i \in {\cal K}_i\subset \RR^{d_i}\ | x \in {\cal F} \}.
\end{equation}

For a point $x\in \RR^{d_1}\times \ldots \times \RR^{d_n}$, its minimum eigenvalue $\lambda_{\min}(x)$ 
is defined as
\[
\lambda_{\min}(x)=
\min(\min_{i\in{\cal LI}} x_i,\ \min_{i\in{\cal SOC}} x_{i0}-\|x_{i1}\|).
\]
A point $x \in {\cal K}$ is called an $\varepsilon$-interior point of ${\cal K}$ if $\lambda_{\min}(x) \geq \varepsilon$.
We define the maximum eigenvalue $\lambda_{\max}(x)$ as
\[\lambda_{\max}(x)=
\max(\max_{i\in{\cal LI}} x_i,\ \max_{i\in{\cal SOC}} x_{i0}+\|x_{i1}\|).
\]
If $x \in {\cal K}$, then 
the following equivalence relation holds betweens $\|\cdot\|_\infty$ and $\lambda_{\max}(\cdot)$:
\begin{equation}\label{equiv}
\|x\|_\infty \leq \lambda_{\max}(x) \leq 2\|x\|_\infty
\end{equation}
%%TT161130
For $x_i\in {\cal K}_i$, the following quantity is called the determinant of $x_i$:
\[
{\rm det}(x_i) \equiv x_i\ \ \ \hbox{if}\ i\in{\cal LI}\ \ \ \hbox{and}\ \ \  
{\rm det}(x_i) \equiv x_{j0}^2-\|x_{j1}\|^2\ \ \ \hbox{if}\ i\in{\cal SOC}.
\]
The determinant of $x\in {\cal K}$ is defined as:
\[
{\rm det}(x) \equiv
\prod_{i\in{\cal LI}} x_i \prod_{j\in{\cal SOC}} (x_{j0}^2 -\|x_{j1}\|^2).
\]
%%TT161130end
Let $e = (e_1; \ldots; e_n)$, where 
\[
e_i=1\ \hbox{if}\ i\in {\cal LI},\ \ e_{i}=(1; 0_{d_i-1})\in \RR\times\RR^{d_i-1}\ \hbox{if}\ i\in{\cal SOC}.
\]
Here $0_{d_i-1}$ denotes the $d_i-1$ dimensional zero vector and we use analogous notation onwards.
We have the following proposition.
\begin{proposition}
Let $x\in \RR^{\bar n}$.  The following relations hold:
\begin{enumerate}
\item
$\lambda_{\min}(x) \geq \varepsilon \Leftrightarrow x \succeq \varepsilon e.$ 
\item $M \geq \lambda_{\max}(x)\Leftrightarrow Me \succeq x.$
\end{enumerate}
\end{proposition}

Let $\varepsilon > 0$.
A point $x\in \RR^{\bar n}$ is called an $\varepsilon$-interior-feasible solution to (P) if the following condition is satisfied: 
\[
A x = 0, \ \ \ \lambda_{\min}(x) \geq \varepsilon,\ \ \ \hbox{or\ equivalently,\ \ \ }Ax = 0,\ \ \ \ x\succeq \varepsilon e. 
\]
%In this paper, we find a feasible solution to the following system.
%\[
%Ax = 0, \ \ \|x\|_\infty = 1, \ x\succeq \varepsilon e.
%\]
We develop a polynomial-time algorithm to find an interior feasible solution to (P) or a 
nonzero feasible solution to (D),
or conclude that no $\varepsilon$-interior feasible solution exists to (P).  The algorithm terminates in
$O(n\log\varepsilon^{-1})$ iterations of a procedure called a basic procedure.  
The basic procedure requires
%%TTT$O(n^2{\bar n}+m^2{\bar n}+nm{\bar n})$ arithmetic operations (assuming that the standard linear algebraic procedures are employed).
%%TT161223$O(n^2{\bar n}\max_i d_i+m^2{\bar n}+m{\bar n}^2)$ arithmetic operations (assuming that the standard linear algebraic procedures are employed).
$O(n^3{\bar n}\max_i d_i+m{\bar n}^2)$ arithmetic operations (assuming that the standard linear algebraic procedures are employed).
Therefore, the algorithm terminates in 
%%TTT$O((n^3{\bar n}+nm^2{\bar n}+n^2m{\bar n})\log\varepsilon^{-1}$)
%%TT161223$O((n^3{\bar n}\max_id_i+nm^2{\bar n}+nm{\bar n}^2)\log\varepsilon^{-1}$)
$O((n^4{\bar n}\max_id_i+nm{\bar n}^2)\log\varepsilon^{-1}$)
 arithmetic operataions.
The basic procedure is a heart of Chubanov's algorithm.
%%TT161130
%A main contribution of this paper is to extend the basic procedure to the second-order cone programs.
%%TT161130end

%%TT161130%%TT161207
In the following, we explain our algorithm in comparison with Chubanov's original algorithm, and discuss the difference between our algorithm 
and Pe\~na and Soheili's algorithm.
Chubanov's algorithm is to find a point in the intersection of a linear space and a unit hypercube,
i.e., the direct product of 0-1 segments.  For simplicity, we assume that the system is interior feasible. 
The algorithm first performs the basic procedure.
The basic procedure either (i) finds an interior feasible solution, or (ii) detects a variable whose value cannot be greater than 1/2 
in the feasible region.  Detection is done by finding a ``cut", a hyperplane to cut off the region where no feasible solution exists. 
Once such ``cut" is found, then the associated coordinate is rescaled by a factor of two so that the hypercube is recovered, 
to continue the same procedure.
%Thus, Chubanov's algorithm finds a cut which shrink the region of existence of a feasible solution, 
%and adjust scaling so that the iterative process can be continued in the same manner.  
In terms of the original 
coordinate, this process is regarded as
generating a series of shrinking convex bodies of the same type, i.e., a hyper-rectangle, which enclose feasible solutions.
%%TT161130end

Now we illustrate our algorithm.  For the ease of understanding, we assume that $\cal K$ is just a single second-order cone
and that there exists an interior feasible solution. We let $\varepsilon = 0$.
%%TT161207
Let ${\cal F}$ be the intersection of the feasible solution to (P) and ${\cal C}_S = \{x \in {\cal K}|\ x_0\leq 1\}$, which we call ``the standard truncated second-order cone.''
The algorithm is to find an interior feasible solution in ${\cal F}$.  To this end, 
first we perform the basic procedure.  The basic procedure either (i) finds an interior solution to (P), or (ii) finds a hyperplane called a ``cut" $\{x|\ w^T x = w^T v\}$.  The cut defines an obliquely truncated second-order cone ${\cal C}_O = \{x|\ w^T x \leq w^T v,\ x \in {\cal K}\}$.
This cut is a natural extension of the one by Chubanov, and is one of the key ideas of this paper.
In virtue of the basic procedure, the set ${\cal C}_O$ contains ${\cal F}$ and has smaller volume than ${\cal C}_S$ at least by a constant factor.  
Thus, if a cut is found, we can shrink the region where the feasible solutions exist.  
Then, ${\cal C}_O$ is transformed to ${\cal C}_S$ with an automorphism transformation of the cone ${\cal K}$.  
We apply the same procedure to the transformed problem, and repeat it over and over.
This way, the algorithm constructs a series of shrinking obliquely truncated second-order cone which contains 
a nonzero feasible solution to (P).
%%TT161207end
It is shown that the volume of obliquely truncated second-order cone converges linearly to zero. Therefore, 
if there exists an interior feasible solution, then shrinkage cannot last forever and the algorithm and the basic procedure 
ends with (i) at a certain point.  
This is a rough sketch of the algorithm, and the idea will be generalized to the multiple cone case in the rest of this paper. 

Interestingly, while the new algorithm has similarity to the ellipsoid method in the sense that it generates
a series of shrinking convex bodies of the same type, it has some flavor of the interior-point method in that
it utilizes the automorphism group of the cone.  
The idea of the cut and the basic procedure is two key concepts in Chubanov's algorithm, and will be extended to the second-order
cone programming case in this paper.  
%%TT161130%%TT161207
As is readily seen from the above explanations of the two algorithms, 
our algorithm is a word-by-word generalization of Chubanov's algortihm.
%%TT161130end

%%TT161130%%%TT161207
Pe\~na and Soheili \cite{Pe16} developed a polynomial-time projection and rescaling algorithm for the symmetric cone feasibility problem.
Their algorithm consists of rescaling step and the basic procedure to find a vector for rescaling, where rescaling procedure is 
inspired by Chubanov's idea.
They measure the progress of the algorithm with a condition number of the system which is essentially the determinant.
The condition number is bounded above by one, and the system whose condition number is closer to one is better conditioned.  In their algorithm,
the condition number is increased by a constant factor at each iteration by rescaling, or the algorithm finds an interior feasible 
solution.  The Chubanov's cut vector is used as an algebraic machinery to rescale the system properly.
Their algorithm plays with scaling (or metric), but does not change the shape of the region on focus.
This makes a remarkable contrast with our approach as we argue below.
%and the geometrical meaning of rescaling is a bit hard to see.
%, but 
%unlike Chubanov's algorithm, the cut is not necessarily used to detect the region of nonexistence of a feasaible solution.

Our algorithm uses the cut to confine the region of existence of the feasible solutions
and generates a series of shrinking convex bodies of the same type containing the feasible region.
In this regard, our algorithm is geometrically intuitive and can be considered as a direct and 
word-by-word extension of Chubanov's algorithm. 
%, while such connection is not so apparent in \cite{Pe16}.
In our algorithm, we measure the progress of the algorithm with the volume of the shrinking area of 
existence of the feasible solutions, which is essentially the determinant.
Thus, the determinant plays crucial roles in the both algorithms and they share some features in common.
%%TT161130

%Chubanov's algorithm is to find a point in the intersection of a linear space and a unit hypercube,
%i.e., the direct product of 0-1 segment.
%The algorithm finds an interior feasible solution or detect a variable whose value cannot be greater than 1/2 
%in the feasible region and then that coordinate is rescaled by one.  
%Similarly, we find a feasible solution to the system of
%
%As a remarkable result in this direction, we mention a very recent work \cite{Pe16} by 
%Pe\~na and Soheili \cite{Pe16} in which they developed an interesting projection and rescaling 
%algorithm for a symmetric cone feasibility problem built upon their rescaled perceptron algorithm and 
%Chubanov's algorithm.  
%Our extension is different from theirs even when restricted to second-order cone programming. 
%While the algorithm in \cite{Pe16} is based on the idea of rescaling the system to improve its condition number,
%our algorithm finds a cut and generates a series of shrinking convex bodies of the same type enclosing the feasible 
%region.  We would say that our algorithm is a direct and a word-by-word extension of Chubanov's algorithm.
%This is nice feature which our extension enjoys.
%Such connection is not so apparent in P\~na and Soheili's algorithm.

%leading to another generalization of Chubanov's idea.  
%%TT161130end

%%

The paper is organized as follows.  In Section 2, we introduce the second-order cone and its automorphism group,
and study some basic properties of the truncated second-order cones.  In Section 3, we discuss an extension of 
Chubanov's fundamental relation in the context of second-order cone programming.  
In Section 4, we extend and analyze the basic procedure.  
%%TTTTT
In Sections 5, we develop the main algorithm.
%
%and 6, two versions of the main algorithm: a basic version and a full version, are developed
%and analyzed.
In Section 6, we make some remarks.  Section 7 is a conclusion.

\medskip
\noindent
{\bf Note added at the Second Revision (January 2017):}

We removed ``Section 6: Full Version" of the paper, 
because we found a flaw in its complexity analysis.
The intension of that section was to reduce the complexity by a factor of $n$ from the algorithm 
in Section 5 by adapting Chubanov's 
elegant idea \cite{Ch13} of initiating a basic procedure using the second last iterate of the preceding basic procedure.
We realized that the argument we made in the previous version does not work.
We feel very sorry to the readers about this mistake, but we consider that the main part of the paper, development of 
an extension of Chubanov's algorithm to second-order cone program and its polynomial-time complexity analysis, yet
survives.

%%TTT
\medskip
\noindent
{\bf Note added at the First Revision (December 2016):}
\begin{enumerate}
\item We corrected an nontrivial error in evaluating complexity of the basic procedure.
In the first version released in November 2016, we conducted analysis assuming that one iteration of the
basic procedure can be done in $O(\bar n)$ arithmetic operations like in the case of linear program.  
But later we realized that the argument in the first version was not correct and 
that one iteration of the basic procedure requires $O(\bar n d_i)$ arithmetic operations.  
At Step 5 of the basic procedure, we compute $P_A \eta$ and this requires $O(\bar n d_i)$ arithmetic operations.  
This affects overall complexity estimate of the entire algorithms.  We corrected them accordingly. 
We feel very sorry for the confusion caused by this flaw.
\item 
We refer the reference \cite{Pe16} and added related considerations in this introduction.
We also updated references and corrected some misleading statements related to \cite{Ro15}.
A few (easily fixable) mathematical errors are also corrected.
\end{enumerate}
%%TT161130end

\section{Preliminary Observations}

We assume that ${\cal K}$ is a $d$-dimensional second-order cone ($d\geq 2$).  For $w, v \in \RR^d$, we define
\[
H(w, v) = \{x\in \RR^d| w^T x \leq w^T v\},
\]
i.e., $H(w,v)$ is the half space in $\RR^d$ whose boundary normal vector is $w$ and $v$ is on the boundary.
The boundary hyperplane of $H(w,v)$ is written as $\partial H(w,v)$.

Let $e=(1;0_{d-1})\in \RR\times\RR^{d-1}$. 
The intersection of the second-order cone and the half space $H(e,e) = \{(x_0; x_1)\in \RR\times \RR^{d-1}|x_0 \leq 1\}$
is referred to as the standard truncated second-order cone (S-TSOC).
We denote by $V_{d}$ the volume of $k$-dimensional S-TSOC.  Its concrete 
formula is:
\[
V_d=\frac{\pi^{(d-1)/2}}{(d-1) \Gamma(\frac{d-1}2+1)},
\]
which is obtained by integrating the volume of $(d-1)$-dimensional hypersphere from the radius $0$ to 1.
%Given a feasible solution $x$ to (P), a oblique truncated second-order cone cover set (Oblique-SOC cover set
%for short)
%is defined as follows:{\cal 
%\[
%x_i \in \Omega_i \subset {\cal K}_i,
%\]
%where $\Omega_i$ is an oblique cone if ${\cal K}_i$ is a second-order cone, and
%$\Omega_i$ is an interval $[0, c]$ if ${\cal K}_i$ is a half-line.

%We will show that for a single second-order cone, with $x=(1, x_1)$ with $\|x_1\|=1-\varepsilon$,
%the minimum volume of oblique SOC cover is bounded by $\sqrt{\varepsilon(2-\varepsilon)}^d$.

%Due to rotational symmetry, we may assume without loss of generality 
%that $x = (1, 1-\varepsilon, 0, \ldots, 0) = (1,(1-\varepsilon)e_1)$.
%There are two possibilities. 

Let $w \in {\rm int}({\cal K})$ and $v \in {\rm int}({\cal K})$.
Then ${\cal K}\cap H(w, v)$ is a non-empty bounded domain 
which is obtained by cutting ${\cal K}$ with a tilted hyperplane.  This set is referred to as an obliquely truncated second-order cone (O-TSOC).
We let
\[
{\cal C}(w,v)=\{x |x \in {\cal K}\cap H(w, v) \}.
\] 
%We assume for later convenience without loss of generaility that 
%$w=(\alpha; \beta)\in \RR\times \RR^{d-1}$ such that $\alpha^2 - \beta^T \beta = 1$. If $w = (w_0; w_1) \in \RR\times\RR^{d-1}$ does not satisfy this condition,then we let $\alpha= w_0/\sqrt{w_0^2-\|w_1\|^2}$ 
%and $\beta=w_1/\sqrt{w_0^2-\|w_1\|^2}$.
With this notation, S-TSOC is written as ${\cal C}(e, e)$.  
%We use ${\cal C}_k$ to represent a $k$-dimensional S-TSOC.

The automorphism group of a cone $\widetilde {\cal K}$ is the set of linear transformations $\tilde G$ such that
\[
\widetilde {\cal K} = \widetilde G\widetilde{\cal K}.
\]
We denote by $Aut(\widetilde {\cal K})$ the automorphism group of $\widetilde {\cal K}$.

Let $w\in {\rm int}({\cal K})$ and $v\in \RR^d$ be such that ${\cal C}(w,v)\not=\emptyset$.
In the following, we show that there exists an element $G$ of $Aut({\cal K})$ which maps
S-TSOC ${\cal C}(e,e)$ to ${\cal C}(w, v)$.  This $G$ plays an important role throughout our algorithm 
development and analysis.

We start with the following statement.

\begin{proposition}\label{auto}
If $G\in \RR^{d\times d}$ satisfies the following conditions:
\begin{enumerate}
\item \begin{equation} \label{gteg}
G^T\left(\begin{array}{cc}1 & 0 \\ 0 & -I \end{array}\right) G = \lambda \left(\begin{array}{cc}1 & 0 \\ 0 & -I \end{array}\right),\ \lambda > 0
\end{equation}
\item There exists a point $\eta\in{\rm int}({\cal K})$ such that $G \eta \in {\rm int}({\cal K})$,
\end{enumerate}
then, ${\cal K}=G{\cal K}$, ${\cal K}=G^T{\cal K}$ and hence $G$ and $G^T$ are elements of $Aut({\cal K})$.
\end{proposition}

\begin{proof}
We fix $\lambda=1$ and show that if the condition 1 with $\lambda=1$ and 
the condition 2 are satisfied, then $G \in Aut({\cal K})$ holds.
This is enough to prove the proposition with general $\lambda >0$.

The main part of the proof is to show that $G$ is invertible and 
$G \partial {\cal K} = \partial {\cal K}$, where $\partial {\cal K}$ is the boundary of ${\cal K}$. Once this is shown, $G{\cal K} = {\cal K}$ readily follows since $G$ is a linear transformation.  After this, we will proceed to demonstrate that $G^T \in Aut({\cal K})$.

We prove that $G$ is invertible and $G \partial {\cal K} = \partial {\cal K}$.
The condition 1 immediately implies that $G$ is an invertible matrix.
Consider the image $G\partial {\cal K}$ where
\[
\partial {\cal K} =\{(x_0; x_1)\in\RR\times\RR^{d-1}|x_0 = \|x_1\|,\ x_0 \geq 0\}.
\]
Let $y(x)=Gx$ and let $x\in \partial {\cal K}$.
Since $\lambda=1$ and hence $y_0(x)^2 - \|y_1(x)\|^2 = x_0^2 - \|x_1\|^2$ holds,we have $y_0(x) = \|y_1(x)\|$ or $y_0(x) = -\|y_1(x)\|$.  We show that the second case never occurs.
Suppose that there exists a point $\hat x \in \partial {\cal K}$ such that $y_0({\hat x}) = -\|y_1({\hat x})\|<0$.  
Consider the line $x(t) = (1-t)\eta + t{\hat x}$,
and let $y(t) = Gx(t)$. Then, since $y(0)=G\eta\in{\rm int}({\cal K})$, we have $y_0(0)>0$  
but $y_0(1)<0$, yielding that $y_0(\hat t)=0$ for some $0\leq\hat t< 1$. 
This implies that $y_1(\hat t)= 0$ as well.  However, 
since $x(\hat t)$ is in the interior of ${\cal K}$, we have $x_0(\hat t)^2-\|x_1(\hat t)\|^2 >0$ whereas 
$y_0(\hat t)=\|y_0(\hat t)\|$ and hence $y_0(\hat t)-\|y_0(\hat t)\|=0$, which is a clear contradiction
to the condition 1.  Thus, whenever $x_0 = \|x_1\|$, we have $y_0(x) = \|y_1(x)\|$.  
This shows that $G\partial {\cal K}\subseteq \partial{\cal K}$.  If we take $G^{-1}:=G$, $G^{-1}$ satisfies the
conditions 1 and 2.  Therefore, we have $G^{-1} \partial {\cal K} \subseteq \partial{\cal K}$ and hence
$\partial {\cal K} \subseteq G \partial {\cal K}$. 
Thus, we have shown $\partial {\cal K}= G \partial {\cal K}$.  
Since $G$ is a linear transformation, 
${\cal K} = G {\cal K}$ follows immediately.

Now we show that $G^T\in Aut({\cal K})$. Let
\[
E =  \left(\begin{array}{cc}1 & 0 \\ 0 & -I \end{array}\right) .
\]
Multiplying $GE$ from the left on the both sides of (\ref{gteg}) and by using $E^2 = I$ and that $G$ is invertible, we 
obtain $G E G^T = E$.  In order to show $G^T e \in {\rm int}({\cal K})$, we use the fact that 
$\eta\in {\rm int}({\cal K})$ if and only if $\eta^T e/\|\eta\|\|e\| > 1/\sqrt{2}$.  We apply this by choosing $\eta= G^T e$.
Since $Ge \in {\rm int}({\cal K})$, we have $e^T Ge/\|Ge\|\|e\| > 1/\sqrt{2}$.  Then it follows that 
$e^T (G^T e)/\|Ge\|\|e\|>1/\sqrt{2}$.  It remains to show that $\|Ge\|=\|G^T e\|$. Since $G^T EG =G E G^T = E$, we have
\[
(Ge)_0^2 - \|(Ge)_1\|^2 =(G^Te)_0^2-\|(G^T e)_1\|^2 =1,
\]
Since $(Ge)_0 = (G^T e)_0$, we have $\|Ge\|=\|G^Te\|$, and we are done.
\end{proof}

%We derive a formula of ${\rm vol}({\cal C}(w,v)).$  To this end, we consider an element $G$ of the
%automorphism group such that
%\[
%{\cal C} =G {\cal C}_k.
%\]
%Then, 
%\[
%{\rm vol}({\cal C}(w,v))={\rm det}(G)V_{d_k}.
%\]

In the following, we will find $G \in Aut({\cal K})$ such that
\[
{\cal C}(w,v) = G {\cal C}(e,e).
\]
Since
\[
{\cal C}(w,v)=H(w,v)\cap{\cal K} = G C(e,e) =G(H(e,e)\cap{\cal K})= (G H(e,e))\cap G{\cal K}=(G H(e,e)) \cap {\cal K},
\]
it is enough to find an element $G$ of the automorphism
group such that $H(w,v) = G H(e,e)$, and since $w^T v > 0$, this amounts to finding $G$ such that
$\partial H(w,v)=G\partial H(e,e)$ 
where $\partial H(w,v)=\{x|w^T x = w^T v\}$
and $\partial H(e,e)=\{u|e^T u = e^T e\}=\{u|u_0 = 1\}$. 

Since $\partial H(e,e) = \{u | u \in (1;u_1),\ u_1\in\RR^{d-1}\}$,  
We have $G \partial H(e,e) = \{x =G u|\ u \in (1; u_1),\ u_1 \in \RR^{d-1}\}$.
The tangent space of $G\partial H(e,e)$ is $T_1 = \{\Delta x= G(0,u_1)|u_1\in \RR^{d-1}\},$
and this should be equal to the tangent space $T_2 = \{\Delta x | w^T \Delta x = 0, \Delta x \in \RR^d\}$ of $\partial H(w,v)$.
Since $T_1=T_2$ should hold,  
\[
w^T G\left(\begin{array}{c}0\\u_1\end{array}\right) = 0\ \forall \ u_1\in \RR^{d-1}.
\] 
Therefore, we have $w^T G =(\lambda; 0)^T=\lambda e^T$ with $\lambda\not=0$ (c.f. $0\not=w\in {\rm int}({\cal K})$).  
This implies that $w = \lambda G^{-T} e$ and equivalently $\lambda e = G^T w$.
Since $G \in  Aut({\cal G})$, so is 
$G^T$, then we have $\lambda e \in {\rm int}({\cal K})$ and hence $\lambda>0$.
Note that $\lambda G^{-T}$ is an element of $Aut({\cal K})$ which maps $e$ to 
$w$.  Since $Ge \in \partial H(w,v)(=\{x | w^T x = w^T v\})$, we have $w^T Ge = w^T v$.  Substituting
$w = \lambda G^{-T} e$ into this formula, we obtain that $\lambda = w^T v$.
In summary, if $G{\cal C}(e,e) = C(w,v)$ and $G\in Aut({\cal K})$,
$G$ should satisfy $w = w^T v G^{-T} e$.  On the other hand, if we can find $G\in Aut({\cal K})$ satisfying
this condition, we have $G{\cal C}(e,e) = C(w,v)$.
In the following, we find $G\in Aut({\cal K})$ satisfying the condition $w = w^T v G^{-T} e$.

Let $\alpha=w_0/\gamma$, $\beta=w_1/\gamma$, where $\gamma= \sqrt{w_0^2 - \|w_1\|^2}$, 
and let
\[
\tilde G =
\left(\begin{array}{cc} \alpha & \beta^T \\ \beta & I + \frac{\beta\beta^T}{1+\alpha}\end{array}\right).
\]
It is not difficult to check that $\tilde G$ satisfies the conditions 1 and 2 in Proposition \ref{auto}, 
being a member of $Aut({\cal K})$.
In particular, we see that ${\rm det}(\tilde G) = 1$ and $\tilde G e = (\alpha;\beta) = w/\gamma$. 
Hence we have $\gamma\widetilde G e = w$. 
Since $w^T v G^{-T} e = w$,
we let $w^Tv G^{-T} = \gamma \widetilde G$ and obtain 
%and $G$ is supposed to belong $Aut({\cal K})$, we obtain $\lambda G^{-T}=\gamma \tilde G$ or, equivalently,
\[
G = w^T v \gamma^{-1}\tilde G^{-1}.
\]
By direct calculation it is easy to confirm that 
\[
\tilde G^{-1} = 
\left(\begin{array}{cc} \alpha & -\beta^T \\ -\beta & I + \frac{\beta\beta^T}{1+\alpha}\end{array}\right).
\]
%To find the scalar constant $\lambda$,  we proceed as follows.
%Since $Ge \in \partial H$ and therefore $w^T v = w^T G e$, we have 
%\[
%w^T G e = w^T (\gamma^{-1}\lambda \tilde G^{-1}) e = w^T v.
%\]
%On the other hand, since 
%\[
%\gamma^{-1} w^T \tilde G^{-1} e = (\alpha \ \beta^T)\left(\begin{array}{c}\alpha\\ - \beta\end{array}\right) = 1,
%\]
%we have $\lambda = w^T v$.  
Therefore, we have
\begin{equation}
\label{transform}
G = \gamma^{-1} w^T v \tilde G^{-1}=(\alpha v_0+\beta^T v_1)\left(\begin{array}{cc} \alpha & -\beta^T \\ -\beta & I + \frac{\beta\beta^T}{1+\alpha}\end{array}\right)
\end{equation}
and
\begin{equation}
{\rm vol}({\cal C}(w,v)) = (\alpha v_0 + \beta^T v_1)^d V_{d}.
\end{equation}
Thus, we obtain the following proposition.
\begin{proposition}\label{p1}
Let $w=(w_0; w_1) \in {\rm int}({\cal K})$, $v \in {\rm int}({\cal K})$, and
consider O-TSOC ${\cal C}(w, v) = \{x|w^T x \leq w^Tv,\ x \in {\cal K}\}$.
Then the matrix
\[ 
G=(\alpha; \beta)^T v\left(\begin{array}{cc} \alpha & -\beta^T \\ -\beta & I + \frac{\beta\beta^T}{1+\alpha}\end{array}\right),
\]
where
\[
\alpha= \frac{w_0}{\sqrt{w_0^2-\|w_1\|^2}}\ \ \ \hbox{and} \ \ \ \beta= \frac{w_1}{\sqrt{w_0^2-\|w_1\|^2}},
\]
maps S-TSOC ${\cal C}(e,e)$ to ${\cal C}(w, v)$, i.e.,
\[
{\cal C}(w,v) = G {\cal C}(e,e)
\]
and
\[
{\rm vol}({\cal C}(w,v)) = \left(\frac{w^T v}{\sqrt{w_0^2-\|w_1\|^2}}\right)^d V_{d}.
\]
\end{proposition}
Suppose that $v\in {\rm int}({\cal K})$ is given, and that we want to find $w=(w_0;w_1)\in \RR\times \RR^{d-1}$ which minimizes the volume of 
${\cal C}(w, v)$.  Since $w\in {\rm int}({\cal K})$, without loss of generality, we may assume that $w=(\alpha; \beta)\in \RR\times \RR^{d-1}$ satisfies
$\alpha^2 -\|\beta\|^2 = 1$.  Furthermore, 
due to rotational symmetry with respect to the 0th axis, 
we assume that, without loss of generality, $v =(\xi_0; \xi_1; 0_{d-2})\in \RR\times\RR\times\RR^{d-2}$.
In order to minimize ${\rm vol}({\cal C}(w,v))$, we just minimize
\[
\min\ \alpha \xi_0 + \beta_1 \xi_1,\ \ \ \hbox{s.t.}\ \alpha^2 - \beta^T \beta =1.
\]
Solving this problem, we obtain that
\[
\alpha = \frac1{\sqrt{1-\eta^2}},\ \ \ \beta = -\frac{\eta}{\sqrt{1-\eta^2}},
\]
where $\eta=\xi_1/\xi_0$.
and the optimal value is:
\[
\sqrt{\xi_0^2 - \xi_1^2}.
\]
%\end{proof}
This consideration is summarized as the following proposition:
\begin{proposition}\label{p2}
Let $v\in {\rm int}({\cal K})$.
A normal vector $w \in {\rm int}({\cal K})$ which minimizes the volume 
${\rm vol}({\cal C}(w,v))$ is given as
\[
w = \left(\begin{array}{c}v_0 \\ -v_1\end{array}\right)
\]
and 
\[
{\rm vol}({\cal C}(w,v)) = (v_0^2 -\|v_1\|^2)^{d/2} V_d.
\]
\end{proposition} 

\begin{corollary} \label{p3}
Let $\hat x \in {\rm int}({\cal K})$.  Then, the minimum volume O-TSOC containing $\hat x$ is given as
\[
{\cal C}\left((\hat x_0; -\hat x_1), \hat x\right),
\]
and hence the minimum volume is given as
\[
{\rm vol}\left({\cal C}\left((\hat x_0; -\hat x_1), \hat x\right)\right)
=(\hat x_0^2 - \|\hat x_1\|^2)^{d/2}V_d.
\]
\end{corollary}

\begin{proof}
Let us denote by ${\cal C}(w,v)$ an O-TSOC satisfying the condition.
Then we can take $v = \hat x$, since, otherwise, we can make a parallel 
shift of the boundary hyperplane $\partial H(w,v)$ until it touches $\hat x$
after the shift.
Now we can apply the previous lemma to obtain the result.
\end{proof}

\begin{proposition}
Let $(x_0; x_1)\in {\rm int}({\cal K})$.
If $\sqrt{x_{0}^2 - \|x_{1}\|^2}\leq \varepsilon$, then
\[
0\leq x_{0} -\|x_{1}\| \leq \varepsilon.
\]
The strict inequality version of this relation also holds.
%, i.e., 
%if we replace two ``$\leq$"s just before $\varepsilon$ with ``$<$"s, the relation holds true.

\end{proposition}
\begin{proof}
The relation obviously holds because
\[
(x_{0}-\|x_{1}\|)^2\leq(x_{0}+\|x_{1}\|)(x_{0}-\|x_{1}\|)\leq x_{0}^2-\|x_{1}\|^2=\varepsilon^2.
\]
\end{proof}

\begin{proposition} \label{eps}
Let ${\cal F}$ and ${\cal F}_i$ be as defined in Section 1.  Let $w \in {\rm int}({\cal K}_i)$ and 
$v \in {\rm int}({\cal K}_i)$, and 
suppose that ${\cal F}_i \subseteq {\cal C}(w, v)$ and ${\rm vol}({\cal C}(w, v))< \varepsilon^{d_i} V_{d_i}$.
Then, there does not exist an $\varepsilon$-interior solution in $\cal F$.
\end{proposition}

\begin{proof}
By contradiction, we assume that there exists an $\varepsilon$-interior solution $\hat x$, say, in $\cal F$.
The $i$-th block $(\hat x_{i0};\hat x_{i1})$ of this solution satisfies
\[
\sqrt{\hat x_{i0}^2 - \|\hat x_{i1}\|^2} \geq \varepsilon
\]
because $\hat x_{i0} + \|\hat x_{i1}\| \geq \hat x_{i0} - \|\hat x_{i1}\| \geq \varepsilon$ (the contraposition of Proposition 2.5).
Since $\hat x_i \in {\cal F}_i \subseteq {\cal C}(w,v)$, ${\cal C}(w,v)$ is an O-TSOC containing $\hat x$
and therefore, in view of Corollary \ref{p3}, 
\[
{\rm vol}({\cal C}(w,v)) \geq (\hat x_{i0}^2 - \|\hat x_{i1}\|^2)^{d_i/2}V_{d_i} \geq \varepsilon^{d_i}V_{d_i}
\]
holds, which is a contradiction to the initial assumption that ${\rm vol}({\cal C}(w,v))<\varepsilon^{d_i} V_{d_i}$.
%This implies that the volume of any O-TSOC contatining $\hat x_i \in {\cal F}_i$ is at least $\delta$,
%contradiction. 
\end{proof}

In the end of this section, we introduce a scaling operation of (P) and (D).
Let $G_i\in Aut({\cal K}_i)$ for $i=1,\ldots, n$, and consider the dual pair of the problems (SP) and (SD):
\[
\hbox{(SP)}\ \ \ \hbox{find}_{(\tx_1;\ldots;\tx_n)}\ 
\sum \tA_i \tx_i = 0,\ \ \tx_i \in {\cal K}_i,\ i=1, \ldots, n,
\]
where $\tA_i = A_i G_i$ for $i=1 ,\ldots, n$ and
\[
\hbox{(SD)}\hbox{\ \ \ find}_{(y, (\ts_1;\ldots;\ts_n))}\ s_i=-\sum \tA_i ^T y,\ \ s_i \in {\cal K}_i,\ i=1, \ldots, n.
\]
(SP) and (SD) are mutually dual and they are obviously equivalent to (P) and (D).  Following interior-point terminology, 
we call (SP) and (SD) ``scaled problems."  In the algorithm developed in this paper, we mostly work with scaled problems
(SP) and (SD).  The original problems (P) and (D) appear only in the beginning and in the end
of the algorithm.

\section{Basic Lemma and its Consequences}

We extend a fundamental relation established by Chubanov (Formula (2), Section 2.1, \cite{Ch13}) 
and its consequences to the second-order cone case.  
For notational convenience, we develop the results in terms of (P) and (D) in Section 1. 
Later we will apply the results in this section to scaled problems (SP) and (SD). 
It is easy to translate the results written in terms of (P) and (D) into the corresponding results in terms of (SP) and (SD).
In the rest of the paper, we denote 
the S-TSOC of the $k$-th block by ${\cal C}_k$, i.e., ${\cal C}_k = {\cal C}(e_{k}, e_{k}) \subset {\cal K}_k \in \RR^{d_k}$. 
The extension of Chubanov's fundamental relation to the second-order cone case is described as follows.

\begin{lemma}\label{basic} (Basic Lemma)
Suppose $x\in{\cal F}$.
Suppose that $y\in{\cal K}$ satisfies the homogeneous inequality
\[
2\sqrt{n}\|P_{A} {y}\| \leq {y}_{k0}%\ \hbox{and}\ {y}_k \in {\cal K}_k
\]
for some index $k$.  Then, if ${\cal K}_k$ is a half-line, we have
\[
{\cal F}_k \subseteq \left[0, \ \frac1{\sqrt{2}}\right],
\]
and if ${\cal K}_k$ is a second-order cone, then, 
\[
{{\cal F}}_k \subseteq H\left(y_k,\frac1{\sqrt{2}}e_k\right)\cap{\cal K}_k = {\cal C}\left(y_k,\frac1{\sqrt{2}}e_k\right).
%x_{k}^T {{{\tilde y}^+}^+}_i \leq \frac1{\sqrt{2}}
\]
In other words, any feasible solution $x\in {\cal F}$ to (P) satisfies $x_k\in [0,1/\sqrt{2}]$ if ${\cal K}_k$ is a half-line, and
$x_k \in {\cal C}(y_k, e_k/\sqrt{2})$ if ${\cal K}_k$ is a second order cone. 
\end{lemma}

\begin{proof}
We give a proof for the case where ${\cal K}_k$ is a second-order cone. The half-line case is analogous and easy.
We have $\|x\| \leq \sqrt{2n}$ if $\|x\|_\infty =1$ and $x\in{\cal F}$.  Therefore,
\[
y_{k}^T x_{k} \leq y^T x = y^T P_{A} x \leq \|x\| \|P_{A} y\| \leq \|x\| \frac{{y}_{k0}}{2\sqrt{n}} \leq \frac1{\sqrt{2} }{y}_{k0}.
\]
Thus, we see that any $x_k\in {\cal F}_k$ is contained in the half space
\[
\left\{x_k\in\RR^{d_k}\left|\right.\ y_k^T \left(x_k - \frac1{\sqrt{2}}e_k\right) \leq 0 \right\} =H\left(y_k, \frac1{\sqrt{2}}e_k\right).
\]
\end{proof}
\noindent
We call $y$ satisfying the condition of lemma as ``a cut generating vector," and
$k$ and $y_k$ are referred to as ``generating index" and ``generating block," respectively.
%\[
%H(w,v) = \{x| w^T x \leq w^T v \}.
%\]
%We have
%\begin{eqnarray*}
%&&\hbox{(i)}\ 0 \leq {\cal F}_k \leq \frac1{\sqrt{2}}\ \ \ \hbox{(if ${\cal K}_k$ is a half-line)},\ \ \  \\
%&&\hbox{(ii)}\ 
%{\cal F}_k \in H\left(y_k,\frac1{\sqrt{2}}e_k\right) \cap {\cal C}_k \ \ \ \hbox{(if ${\cal K}_k$ is a 
%second-order cone)}.
%\end{eqnarray*}
%In the following, we focus on (ii), the second-order cone case. Observe that, since ${\cal C}_k \subset {\cal K}_k$,
%\[
%{\cal F}_k \subseteq H\left(y_k,\frac1{\sqrt{2}}e_k\right) \cap {\cal C}_k \subseteq
%H\left(y_k,\frac1{\sqrt{2}}e_k\right) \cap {\cal K}_k
%={\cal C}\left(y_k,\frac1{\sqrt{2}}e_k\right).
%\]
%Thus, ${\cal F}_k$ is enclosed by the O-TSOC ${\cal C}(y_k,\frac1{\sqrt{2}}e_k)$.
%If the volume of ${\cal C}(y_k,\frac1{\sqrt{2}}e_k)$ is smaller than the volume of ${\cal C}_k$,
%the domain of existence of ${\cal F}_k$ is confined to a smaller region.  Furthermore, an 
%O-TSOC is convertible to a S-TSOC by applying automorphism transformation of ${\cal K}_k$.
%Thus, if we can continue this process repeatedly as we reduce the volume of the region of existence of ${\cal F}_k$.

Suppose that ${\cal F}_k \in {\cal C}_k$, and a cut generating vector $y$ with generating index $k$ is found.
In the rest of this section, 
%derive an O-TSOC enclosing ${\cal F}_k$ with some reduction of volume from ${\cal C}_k$, 
%assuming that ${\cal F}_k \subseteq {\cal C}_k$ and a cut generating vector $y$ is given.
we construct an O-TSOC ${\cal C}(w,v)$  which encloses ${\cal F}_k$ and with smaller volume than ${\cal C}_k$ 
by choosing appropriate $w \in \RR^{d_k}$ and $v \in \RR^{d_k}$.
Specifically, we find $w, v$ satisfying the following two conditions:
%%TT161130
\begin{eqnarray}
&&
{\cal F}_k \subseteq 
H\left(y_k,\frac1{\sqrt{2}}e_k\right) \cap {\cal C}_{k} \subseteq H(w,v) \cap{\cal K}_k ={\cal C}(w,v), \label{3}\\
&& 
{\rm vol}\left(H\left(y_k,\frac1{\sqrt{2}}e_k\right) \cap {\cal C}_{k}\right) \leq {\rm vol}(H(w,v) \cap{\cal K}_k)
={\rm vol}({\cal C}(w,v)) \nonumber \\ 
&&\ \ \ \ \leq 0.96^{d_k}{\rm vol}({\cal C}_k)=0.96^{d_k} V_{d_k} .  \label{4}
\end{eqnarray}
%%TT161130end
If these conditions are satisfied, the pair $(w,v)$ is called ``a cut." 
We also use the term ``cut" for the hyperplane $\partial H(w,v)$.  

\begin{figure}[htbp]
	\begin{center}
		\includegraphics[width=10cm,bb=0 0 551 323]{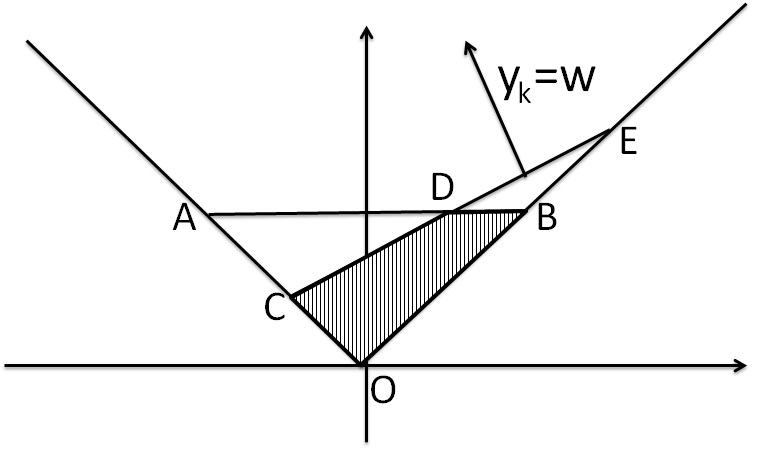}
	\end{center}
	\caption{Case 1}
%\end{figure}
%\begin{figure}[htbp]
	\begin{center}
	\vspace{1cm}\hspace{1.5cm}
		\includegraphics[width=12cm,bb=0 0 637 447]{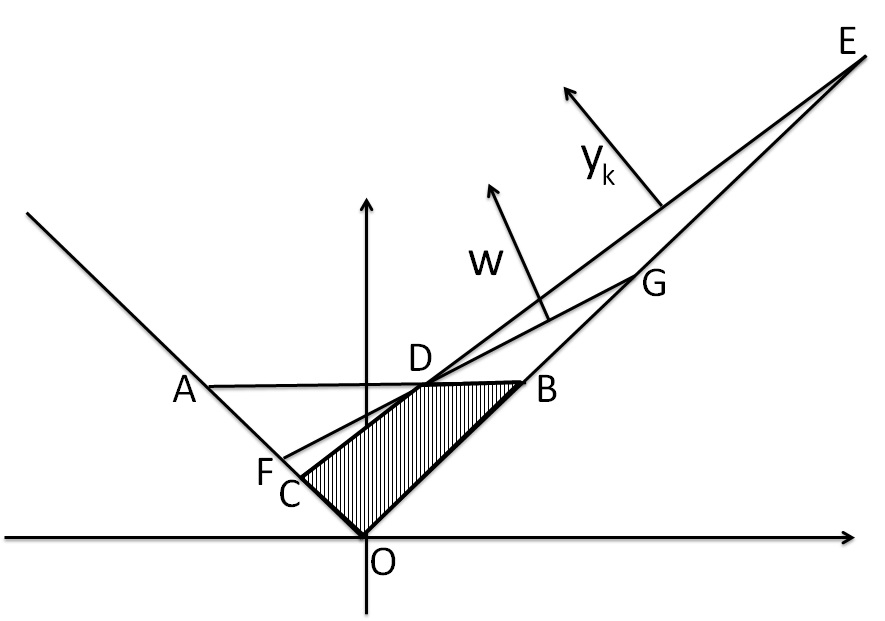}
	\end{center}
	\caption{Case 2}
\end{figure}

We illustrate the situation in Figure 1 for the case where the dimension of ${\cal K}_k$ is two.
In the beginning, we only know that ${\cal F}_k$ is enclosed 
in the triangle AOB ($={\cal C}_k$).
Let a cut generating vector $y$ is given with generating index $k$.   
The quadrangle COBD, which is the intersection of the two triangles AOB and COE $(={\cal C}_k \cap H(y_k,e_k/\sqrt{2}))$, 
is the reduced area where
${\cal F}_k$ is still enclosed. (The triangle ACD is the area which was ``cut off.")
Here, the line CE can be taken as a cut.
Figure 1 intuitively shows that the triangle OCE satisfies the conditions (\ref{3}) and (\ref{4}),
since it encloses the quadrangle COBD and the area of the triangle COE is smaller than the triangle AOB 
%according to the following correspondence:
%\[
%
%\]
%In the following, we find a triangle (O-TSOC) enclosing the quadrangle COBD 
%which is smaller than the original triangle OAB.
%From the figure, we immediately see that the triangle OCE fits our purpose.

Now we generalize this intuitive observation in a more quantitative manner to construct an O-TSOC we are aiming at.
We branch into two cases: (Case 1) The angle between $y_k$ and the center axis $e_k$ is small;
(Case 2) The angle between $y_k$ and $e_k$ is large.

\medskip
\noindent {\bf (Analysis of Case 1)}

In general, since ${\cal C}_k\subset {\cal K}_k$, 
we can take, as was suggested in the above, 
$w=y_k$ and $v = e_k/\sqrt{2}$.  This means that we just use
\[
H\left(y_k,\frac1{\sqrt{2}}e_k\right)\cap {\cal K}_k ={\cal C}\left(y_k,\frac1{\sqrt{2}}e_k\right)
\]  
as a confined enclosing area for ${\cal F}_k$ satisfying (\ref{3}) and (\ref{4}).
%Without loss of generality, by
%rescaling, if necessary, we assume that $y_{k0}^2 - \|y_{k1}\|^2=1$.  
%Note that
%the norm $\|y_k\|$ goes to infinity as it approaches to the boundary.
The volume of ${\cal C}(y_k,e_k/\sqrt{2})$ is, by letting $w=y_k, v = e_k/\sqrt{2}$ in Proposition{\ref{p2}}, given as
\[
{\rm vol}\left({\cal C}\left(y_k,\frac1{\sqrt{2}}e_k\right)\right)=\left(\frac1{\sqrt{2}}\frac{y_{k0}}{\sqrt{y_{k0}^2-\|y_{k1}^2\|}}\right)^k V_{d_k}.
\]
If $y_{k0}^2/(y_{k0}^2-\|y_{k1}\|^2) < 2$, the volume is ensured to decrease.
For later convenience, let ${\hat y}_k= y_k/ y_{k0}$.  We have ${\hat y}_{k0}=1$, and let $\eta =\|{\hat y}_{k1}\|$.
Since $y_k\in {\cal K}_k$ and so is ${\hat y}_k$, the range of $\eta$ is $[0,1]$. 
In terms of $\eta$, the condition $y_{k0}^2/(y_{k0}^2-\|y_{k1}\|^2) < 2$ is written equivalently as $\eta < 1/\sqrt{2}$.
Under this condition, the reduction ratio of the volume is written as:
\[
g_1(\eta) = \frac1{\sqrt{2(1-\eta^2)}^k}.
\]
Observe that this is a monotonically increasing function.

The above idea does not work for $\eta \geq 1/\sqrt{2}$.
See Figure 2.  Observe that $y_k$ is almost parallel to the edge of the cone.
The vertex $E$ is seen further than before.
The quadrangle COBD contains ${\cal F}_k$.
The triangle COE is larger than the original triangle AOB though it contains
the quadrangle COBD. The triangle COE cannot be used to confine the existing area of ${\cal F}_k$ this time.

We consider the following approach to deal with this case.  
We continue explanation with Figure 2.  This time, 
we generate a supporting line (segment) FG which touches the quadrangle COBD at the vertex D, and enclose 
the quadrangle COBD with the triangle FOG. The line FG is chosen so that the triangle FOG contains the quadrangle COBD and 
its volume gets smaller than the original triangle AOB.
We already developed a formula to find a line which goes through $D$ and minimizes the area of the triangle FOG, 
see Proposition \ref{p3}.  We also need to check that the resulting line {\em does not} intersect the quadrangle
COBD.

\medskip
\noindent
{\bf (Analysis of Case 2)}

Consider a half-space $H(w,v)$ which contains $H(y_k,\frac1{\sqrt{2} }e_k)\cap H(e_k,e_k)$ such that 
its boundary $\partial H(w,v) =\{x|w^T x =v\}$ is a supporting hyperplane to 
$H(y_k,\frac1{\sqrt{2}}e_k)\cap H(e_k,e_k)$. Obviously, we have
\[
H\left(y_k,\frac1{\sqrt{2}}e_k\right)\cap {\cal C}_k = 
H\left(y_k,\frac1{\sqrt{2}}e_k\right)\cap H(e_k,e_k)\cap{\cal K}_k
\subseteq H(w,v) \cap {\cal K}_k ={\cal C}(w,v).
\]

Since $\partial H(w,v)$ is a supporting hyperplane to $H(y_k,\frac1{\sqrt{2}}e_k)\cap H(e_k,e_k)$,
without loss of generality, we assume that 
\begin{equation} \label{vinpartial}
v\in \partial H(y_k,\frac1{\sqrt{2}}e_k)\cap \partial H(e_k,e_k).
\end{equation}
Let ${\hat y}_k = y_k/y_{k0}$, then ${\hat y}_{k0} = 1$ and $\|{\hat y}_{k1}\| < 1$. 
We let $\eta=\|{\hat y}_{k1}\|$ as before.  We assume that $v$ is written as follows
\begin{equation}  \label{r0}
v = - \tilde \alpha \hat y_k + \tilde \beta e_k,
\end{equation}
and try to find $w$ and $v$ satisfying the condition such that
\[
{\rm vol}({\cal C}(w,v))<0.96^{d_k} V_{d_k},\ \ \ H\left(y_k,\frac1{\sqrt{2}}e_k\right)\cap H(e_k,e_k)\subseteq H(w,v)
\]
under a certain condition on $y_k$.

Due to (\ref{vinpartial}), we have
\[
{\hat y}_k^T \left(v - \frac{e_k}{\sqrt{2}}\right) = 0,\ \ \ e^T (v - e) = 0.
\]
Therefore, taking ${\hat y}_{k0}=1$ into account, we obtain
\[
-\tilde\alpha(1+\|{\hat y}_{k1}\|^2)+\tilde\beta=\frac1{\sqrt{2}},\ \ \ -\tilde\alpha + \tilde\beta =1.
\]
Solving this with respect to $\tilde \alpha$, we have
\begin{equation} \label{8a}
\tilde\alpha = \frac1{\|{\hat y}_{k1}\|^2}\left(1-\frac1{\sqrt{2}}\right),\ \tilde \beta = 1 + \tilde\alpha.
\end{equation}
Note that $\tilde \alpha > 0$.
Plugging (\ref{8a}) into (\ref{vinpartial}), we obtain
\[
v({y}_k) = \left(1; -\frac1{\|{{\hat y}}_{k1}\|^2}\left(1-\frac1{\sqrt{2}}\right){{\hat y}}_{k1}\right) = (1; - \tilde\alpha{{\hat y}}_{k1}).
\]
Then Proposition \ref{p2} yields that the minimum volume O-TSOC ${\cal C}(w,v)$ is obtained
by taking
\[
w({y}_k) = \left(1; \frac1{\|{{\hat y}}_{k1}\|^2}\left(1-\frac1{\sqrt{2}}\right){{\hat y}}_{k1}\right) = (1; \tilde\alpha{{\hat y}}_{k1})
\]
and 
\[
{\rm vol}({\cal C}(w(y_k), v(y_k)))= g_2(\eta({\hat y}_k)), \ \ \hbox{where}\ \ g_2(\eta) \equiv
\left(1 - \left(\frac1{\sqrt{2}}-1\right)^2\frac1{\eta^2}\right)^{d_k/2}
V_{d_k}.
\]
Observe that $g_2$ is a monotonically increasing function of $\eta$ whose value is positive in the interval
$1-1/\sqrt{2} < \eta \leq 1$.  It is easy to see that $g_2(1) \leq 0.96^{d_k}V_{d_k}$.

Now we examine the condition that $\partial H(w, v)$ defines a supporting hyperplane of $H(e_k,e_k)\cap H(y_k,e_k/\sqrt{2})$.
Since $v \in \partial H(e_k, e_k)$ and $v \in \partial H(y_k, e_k/\sqrt{2})$, 
a necessary and sufficient condition for $\partial H(w, v)$ to be the supporting hyperplane 
is that $w$ is written as a nonnegative combination of $y_k$ and $e_k$.  Since
$(0;{\hat y}_{k1}) = {\hat y}_k - e_k$,
\[
w = \tilde \alpha ({\hat y}_k - e_k) + e_k = \tilde \alpha {\hat y}_k + (1-\tilde \alpha) e_k.
\]
Thus, $w$ is on the line connecting ${\hat y}_k$ and $e_k$, and can be represented as 
a conic combination of ${\hat y}_k$ (or equivalently $y_k$) and $e_k$ if and only if
$0 \leq \tilde \alpha \leq 1$, i.e., 
\[
0 \leq  \frac1{\eta^2}\left(1-\frac1{\sqrt{2}}\right) \leq 1.
\]

\medskip

The analysis so far is summarized as follows:

\begin{enumerate}

\item Suppose that $\eta \leq 1/\sqrt{2}$.  Then, by letting $w = y_k$ and $v = \frac1{\sqrt{2}}e_k$,
O-TSOC ${\cal C}(w, v)$ encloses ${\cal F}_k$ and its volume is bounded by
\[
g_1(\eta) = \left(\frac1{\sqrt{2(1-\eta^2)}}\right)^{d_k} V_{d_k}
\]
The function $g_1(\eta)$ is well-defined in the interval $[0,1)$ 
and is monotonically increasing. In particular, if $\eta=0.6 \leq 1/\sqrt{2}$, we have
$g_1(0.6)=\frac{V_{d_k}}{\sqrt{1.28}^k} < 0.96^{d_k} V_{d_k}$.

\item Suppose that  $\sqrt{1 - \sqrt{1/2}} \leq \eta \leq 1$.  Then by letting
\[
w= \left(1; \frac1{\eta^2}\left(1-\frac1{\sqrt{2}}\right){{\hat y}}_{k1}\right)
\ \ \ \hbox{and}\ \ \ 
v= \left(1; -\frac1{\eta^2}\left(1-\frac1{\sqrt{2}}\right){{\hat y}}_{k1}\right),
\]
O-TSOC ${\cal C}(w, v)$ encloses ${\cal F}_k$ and its volume is bounded by
\[
g_2(\eta) \equiv \left(1 - \frac{1.5-\sqrt{2}}{\eta^2}\right)^{d_k/2} V_{d_k}.
\]
as long as $\sqrt{1 - \sqrt{1/2}} \leq \eta \leq 1$.  In particular, $\eta=0.6$ is in the 
interval.
The function is monotone increasing, and we have $g_2(1)\leq 0.96^{d_k} V_{d_k}.$

\end{enumerate}

Therefore, given the cut generating vector $y$ with generating index $k$, if we determine $w$ and $v$ according to the 
rule that
\begin{enumerate}
\item if $0 \leq \eta \leq 0.6$, then take $w$ and $v$ as in the item 1 above,
\item if $0.6 < \eta \leq 1$, then take $w$ and $v$ as in the item 2 above.
\end{enumerate}
Then, O-TSOC ${\cal C}(w, v)$ encloses ${\cal F}_k$ and 
the bound
\[
{\rm vol}({\cal C}(w, v)) \leq 0.96^{d_k} V_{d_k},
\]
is ensured.  Finally, Proposition \ref{p1} yields that an element of the automorphism group which maps 
${\cal C}_k$ to ${\cal C}(w, v)$ is:
\[ 
G=(\alpha;\beta)^T v\left(\begin{array}{cc} \alpha & -\beta^T \\ -\beta & I + \frac{\beta\beta^T}{1+\alpha}\end{array}\right)
\]
where $\alpha = w_{0}/\sqrt{w_{0}^2 - \|w_{1}\|^2}$ and $\beta = w_{1}/\sqrt{w_{0}^2 - \|w_{1}\|^2}$.

\section{Basic Procedure and its Analysis}

In this section, we explain and analyze the basic procedure which is a direct extension of Chubanov's.
The basic procedure deals with a pair of the dual problems $A x=0,\ x\succeq 0$ and $y = -A^T u,\ y \succeq 0$, and 
finds either a primal interior solution,  or dual nonzero solution, or a cut generating vector.
In the procedure, the iterate $y$ satisfying $e^T y = 1,\ y\succ 0$ is updated every iteration.  
As will be discussed later, the iteration complexity estimate is based on the fact that 
the quantity $1/\|P_A y\|^2$ increases at least by 1/2 at each iteration.
On the other hand, we can show that $y$ is a cut generating vector if $1/\|P_A y\|^2 \geq 4n^3$.
Then, the basic procedure is ensured to terminate in $O(n^3)$ iterations regardless of the choice of 
initial value of $y$.  Before we proceed, we make two important comments:
\begin{enumerate}
\item The basic procedure is mainly applied 
to a scaled system (SP) and (SD).  But we describe the procedure and conduct analysis just for (P) and (D) to avoid
that the notation gets too heavy.

%%TTTTT
%\item The basic procedure is called many times in the main algorithms explained in the following sections.  
%We will introduce two version 
%of the main algorithm: a basic version and a full version with reduced complexity.
%While the basic version is simpler and easier to understand, the full version enjoys better complexity.  The major difference 
%between the two is the initial value of basic procedure.  In the basic version, the initial value of the basic procedure 
%is fixed, however, one can show that the complexity of the main algorithm can be further reduced 
%if we start a basic procedure from an initial value constructed from the last iterations of the preceding basic procedure.
%In order to implement this idea, we need to modify slightly the basic procedure in the full version.  The changed parts are 
%enclosed in bold squared bracket ${\bf [}\ldots {\bf ]}$ and written in italic.  The reader may skip those ``bracketed" part until he/she
%reads Section 6: Main Algorithm: Full Version with Reduced Complexity later.

%%TTT
\item 
%%TTTTT Like i
In Chubanov's algorithm, one iteration of his basic procedure requires just $O(\bar n)$ arithmetic operations
though it computes projection of a vector to ${\rm Ker}(A)$ and appears to require $O(\bar n^2)$ arithmetic operations.
In our case, the complexity of one iteration of the basic procedure is a bit higher and $O(\bar n \max_i d_i)$ arithmetic
operations, because the second-order cone is a bit more complicated than linear inequalities. 
We mention that Pe\~na and Soheili \cite{Pe16} extends the basic procedure to general symmetric cone programming.

\end{enumerate}

\medskip
\noindent
{\bf The Basic Procedure}

\noindent
{\bf Input} Matrix $A$ and vector $y$ such that $e^T y = 1$, $y\succ 0$,

\noindent
{\bf Output} One of the followings :
(i) A cut generating vector $y$ and its generating index $k$,
(ii) Solution $x$ to $A x = 0$, $x\succ 0$; (iii) Solution $y$ to $y=A^T u \succeq 0, \ y \not= 0$,

\medskip
\noindent
{\bf Procedure}

\vspace{-2mm}
\begin{enumerate}

%%TTT\item For all $i=1, \ldots, n$, compute $w_i \equiv P_A \tilde e_i$, where $\tilde e_i$ is $e_i$ for the $i$th block and zero for other blocks.  Compute $z=P_A y$.
\item Compute $P_A$ and $z=P_A y$.

\item Check termination conditions (based on $z$):
\begin{enumerate} 
\item If $z = 0$, then $y$ is dual interior feasible. Return ``(iii)" and $y$.
\item If $z\succ 0$, then, $z$ is primal interior feasible. Return ``(ii)" and $z$.
\item If $2\sqrt{n}\|P_A y\| \leq y_{k0}$ holds for some $k$, return ``(i)", and, $y$ and $k$ as a cut generating vector
and $k$ a generating index, respectively.
\item If (a)--(c) does not hold, then, go to Step 3.
\end{enumerate}

\item Since the conditions (a) and (b) do not hold, $z\not=0$ and $z\not\in {\rm int}({\cal K})$.
Therefore, there exists an index $i$, say, such that $z_i\not=0$ and $z_i\not\in {\rm int}({\cal K}_i)$ hold.
In the following, we construct $\eta_i \in {\cal K}_i$ such that $\eta_{i0}=1,\ \eta_i^T z_i \leq 0.$
%\noindent
%{\bf Case (i)} {\bf (Procedure)}
%We pick an element from $\eta_k \in {\rm int}({\cal K}_k)$ such that $\eta_{k0}=1,\ \eta_k^T z_k =0.$
%If we let $\tilde z_k = z_k/\|z_k\|$ and 
%\[
%\eta_k = \tilde z_k+t (e_k - \tilde z_k), \ \ \ t = \frac{\|\tilde z_k\|^2}{1-\|\tilde z_k\|^2}.
%\]
\begin{itemize}
\item 
If ${\cal K}_i$ is  a half-line, then, we set $\eta_i=e_i$. ($\eta_i$, $z_i$ are scalers and $z_i\leq 0$.)
\item 
If ${\cal K}_k$ is a second-order cone,
such $\eta_i$ is computed as follows.  
\begin{itemize}
\item If $z_{i0} \leq 0$, then we let $\eta_i = e_i$; 
\item If
$z_{i0}> 0$, then, we let
\[
\eta_i = e_i - \frac{{\hat z}_i-e_i}{\|{\hat z}_i - e_i\|},
\]
where $\hat z_i = z_i/z_{i0}$ (${\hat z}_{i0} = 1$).
In this case, $\|\hat z_i - e_i\|\geq 1$ holds because $\hat z_i\not\in {\rm int}({\cal K}_i)$.  Therefore, we have
\[
\hat z_i^T\eta_i = \hat z_i^T e_i - \frac{\|\hat z_i - e_i\|^2 + e_i^T(\hat z_i - e_i)}{\|\hat z_i - e_i\|}
= 1 - \|\hat z_i - e_i\| \leq 0.
\]
\end{itemize}
\end{itemize}

\item Let $\eta = (0, \ldots,0, \eta_i,0, \ldots, 0)$.   Then we have $e^T \eta = 1$ and $\eta^T z \leq 0$.

%%TTT\item Let $p = P_A \eta$.  Since $\eta$ is a linear combination of $z$ and $\tilde e_k$, we can compute $p$ in $O({\bar n})$ arithmetic operations, since $w_k = P_A \tilde e_k$ is computed in advance.
\item Let $p = P_A \eta$.  Computation of $p$ requires $O(\bar n d_i)$ arithmetic operations, since $P_A$ is already computed and $\eta$ contains $d_i$ nonzero elements.

\item Check termination conditions (based on $p$):
\begin{enumerate} 
\item If $p = 0$, then $\eta$ is dual feasible. Return ``(iii)" and $\eta$.
\item If $p \succ 0$, then, $p$ is primal interior feasible. Return ``(ii)" and $p$.
\item If neither of (a) nor (b) holds, go to Step 7.
\end{enumerate}
\item Construct a new iterate $\tilde y$ as follows:
\[\
{\tilde y} = \alpha y + (1-\alpha) \eta,\ \ {\tilde z}= \alpha z + (1-\alpha) p,\ \ \alpha=\frac{p^T(p - z)}{\|z-p\|^2}.
\]
Note that $p\not=0$ and $z\not=0$ are ensured, and that 
$\alpha$ is chosen so that $\tilde z = P_A{{\tilde y}}$ is closest to the origin. Then it follows that $\alpha$ 
is positive as is discussed below.
Since $e^T y= 1$ and $e^T \eta = 1$, we have $e^T {{\tilde y}} = 1$.
Since $\alpha > 0$, ${\tilde y} \succ 0$.
So we continue iteration by letting $y:={{\tilde y}}$, $z:=\tilde z$ and going to Step 2.
\end{enumerate}

\noindent
{\bf (Analysis of change of $1/\|{z}\|^2$)}

We show that $1/\|z\|^2$ increases by at least 1/2 per iteration.
Observe that 
\[p^T {z} = (P_A \eta)^T z = \eta^T (P_A z) = \eta^T {z} = \eta_i^T z_i \leq 0.\]
Furthermore, since 
\[
\|{\tilde z}-p\|^2 = \|z\|^2 + \|p\|^2 - 2z^T p
\]
and neither $z$ nor $p$ is zero, $\alpha \in (0,1)$. Therefore, we have ${\tilde y} \succ 0$.

Letting ${\tilde z} = P_A {\tilde y}$, we have
\[
{\tilde z} =p + \alpha (z - p).
\]
Therefore, 
\[
\|{\tilde z}\|^2 = \alpha^2\|z - p\|^2+2\alpha p^T(z- p) + \|p\|^2.
\]
Substituting the concrete formula of $\alpha$ into the above and using $p^T z \leq 0$, we obtain
\[
\|{\tilde z}\|^2 = \|p\|^2 - \frac{(p^T(z - p))^2 -(z^T p)^2}{\|z-p\|^2}  = \frac{\|z\|^2\|p\|^2 - (z^T p)^2}{\|z\|^2+\|p\|^2-2z^T p} \leq \frac{\|p\|^2\|z\|^2}{\|z\|^2+\|p\|^2}.
\]
Since $P_A$ is a projection matrix, we have $\|p\|^2\leq\|\eta\|^2 \leq 2$.
This implies that
\[
\frac1{\|{\tilde z}\|^2}\geq \frac1{\|z\|^2} + \frac1{\|p\|^2} \geq \frac1{\|z\|^2} + \frac12.
\]

\noindent
{\bf Complexity Analysis of the Basic Procedure}

Now we analyze overall complexity of the basic procedure.
%%TTTPrior to the iteration of the basic procedure, we compute $P_A y$ and $P_A \tilde e_i$ for each $i$.  This 
%%TTTrequires $O({\bar n}m^2+{\bar n}mn)$ arithmetic operations.
Prior to the iteration of the basic procedure, we compute $P_A$ and $P_A y$.  This
%%TT161223 requires $O({\bar n}m^2+{\bar n}^2 m)$ arithmetic operations.
requires $O({\bar n}^2 m)$ arithmetic operations.
%%TTTIn one iteration of the basic procedure, we need to compute $p = P_A\eta$.  This can be done in $O({\bar n})$ arithmetic operations as explained in the previous section.
In one iteration of the basic procedure, we need to compute $p = P_A\eta$.  This can be done in $O({\bar n}\max_i d_i)$ arithmetic operations as explained in the previous section.

We analyze that the number of iterations of the basic procedure is $O(n^3)$.
Recall the condition that $y$ is a cut generating vector with generating index $i$ is $2\sqrt{n} \|z\| \leq y_{i0} \leq \|y\|_\infty$.
since $e^T y =1$ and $y \succ 0$, we obtain that $1/n \leq \|y\|_\infty$.  Therefore, if 
\[
4n^3 \leq \frac1{\|z\|^2},
\]
then $y$ associated with $z = P_A y$ is a cut generating vector.
As we already analyzed, at each step of the basic procedure $1/\|z\|^2$ increases by 1/2.  Therefore, 
in $O(n^3)$ iterations of the basic procedure, we find a cut generating vector or a primal interior feasible solution or dual nonzero feasible solution. Since one iteration of the basic procedure
%%TTTrequires $O({\bar n})$ arithmetic operations, the basic procedure terminates in $O(n^3{\bar n} + m^2{\bar n} + mn{\bar n})$ arithmetic operations.
%%TT161223 requires $O({\bar n})$ arithmetic operations, the basic procedure terminates in $O(n^3{\bar n}\max_i d_i + m^2{\bar n} + m{\bar n}^2)$ arithmetic operations.
requires $O({\bar n}\max_i d_i)$ arithmetic operations, the basic procedure terminates in $O(n^3{\bar n}\max_i d_i + m{\bar n}^2)$ arithmetic operations.

%%TTTTT
\section{Main Algorithm}

We are ready to describe the main algorithm. 
This algorithm (i) finds an interior feasible solution to (P),
(ii) finds a nonzero solution to (D), or (iii) concludes that there exists no $\varepsilon$-interior
feasible solution in $O(n\log\varepsilon^{-1})$ iterations of the basic procedure, 
%%TT161223 where the basic procedure requires $O(n^3{\bar n}{\max_i d_i}+m^2{\bar n}+nm{\bar n})$ arithmetic operations.
where the basic procedure requires $O(n^3{\bar n}{\max_i d_i}+m{\bar n}^2)$ arithmetic operations.
%%TT161223 Thus, the overall arithmetic operaqtions of the algorithm presented in this section is $O(n(n^3{\bar n}+m^2{\bar n}+nm{\bar n}))\log\varepsilon^{-1})$.
Thus, the overall arithmetic operations of the algorithm presented in this section is $O(n(n^3{\bar n}{\max_i d_i}+m{\bar n}^2))\log\varepsilon^{-1})$.
%%TTT In Section 6, we reduce this complexity to $O(n(n^2{\bar n}+m^2{\bar n}+nm{\bar n})\log\varepsilon^{-1})$.
%%TT161223 In Section 6, we reduce this complexity to $O(n(n^2{\bar n}\max_i d_i +m{\bar n}+m{\bar n}^2)\log\varepsilon^{-1})$.
%In Section 6, we reduce this complexity to $O(n(n^2{\bar n}\max_i d_i +m{\bar n}^2)\log\varepsilon^{-1})$.

\medskip\noindent
{\bf The Main Algorithm}%%TTTTT: Basic Version}

\noindent
{\bf Input} A matrix $A$ and a cone ${\cal K}$ which is the direct product of second-order cones and half-lines,

\noindent
{\bf Output} One of the followings :
(i) Solution $x$ to $A x = 0$, $x\succ 0$; (ii) Solution $y$ to $y=-A^T u \succeq 0, \ y \not= 0$,
(iii) Declare ``No $\varepsilon$-interior solution to $Ax = 0$, $x\succeq 0$."

\medskip
\noindent
{\bf Algorithm}

\vspace{-2mm}
\begin{enumerate}

\item Let $t:=0$, $v_i :=1$, $i=1, \ldots, n$, $A^{(0)} := A$, $y^{in}:=e/n$,
$M^{(0)}=I$, $A^{(0)}=A$;

\item Call the Basic Procedure (BP) by setting $A^{(t)}$ and $y^{in}$ as Input (See Section 4).

\begin{itemize}

\item[(C1)] If (BP) returns a cut generating vector $y$ and generating index $k$, 
then proceed to Step 3.

\item[(C2)] If (BP) returns an interior solution $\tilde x$ to $A^{(t)} \tilde x = 0,\ \tilde x\succeq 0$, then we let $x:=M^{(t)} \tilde x$
and return $x$ as an interior solution to (P).

\item[(C3)] If (BP) returns a nonzero solution $y$ to $y = -(A^{(t)})^T u,\ y\succeq 0$, $y\not=0$, then return 
$y$ as a nonzero solution to (D).

\end{itemize}

\item In the case of (C1),
\begin{itemize}
\item if ${\cal K}_k$ is a half space then set $G =1/\sqrt{2}$.
\item if ${\cal K}_k$ is a second-order cone and $\eta \leq 0.6$ where $\eta= \|y_{k1}\|/y_{k0}$, then set $w = (y_{k0}; y_{k1})$, $v = e_k/\sqrt{2}$,
and construct $G$ according to Prposition \ref{p1} as an automorphism transformation of ${\cal K}_k$ mapping ${\cal C}_k$ to ${\cal C}(w,v)$.
\item if ${\cal K}_k$ is a second-order cone and $0.6< \eta \leq 1$ where $\eta= \|y_{k1}\|/y_{k0}$, then set 
\[
w= \left(1; \frac1{\eta^2}\left(1-\frac1{\sqrt{2}}\right)\frac{y_{k1}}{y_{k0}}\right)
\ \ \ \hbox{and}\ \ \ 
v= \left(1; -\frac1{\eta^2}\left(1-\frac1{\sqrt{2}}\right)\frac{y_{k1}}{y_{k0}}\right),
\]
and construct $G$ according to Proposition \ref{p1} as an automorphism transformation of ${\cal K}_k$ mapping ${\cal C}_k$ to ${\cal C}(w,v)$.
\end{itemize}
\item We set
\[
A_k^{(t+1)} := A_k^{(t)} G,\ v_k := {\rm det}(G) v_k,\ M_k^{(t+1)} = M_k^{(t)} G.
\]   
Regarding other blocks than $k$, we let $A_i^{(t+1)}= A_i^{(t)}.$
\item If $v_k \leq \varepsilon^{d_k}$, then, we conlude that there is no $\varepsilon$ interior feasible solution to (P).  
Otherwise, we set $t:=t+1$ and return to Step 2.

\end{enumerate}

\noindent
{\bf Overall Complexity Analysis}

Now we analyze the complexity of the main algorithm.
In the beginning of the algorithm, ${\cal F}_k$ is enclosed in ${\cal C}_k$ for all $k$.
The algorithm terminates

\medskip\noindent (i) In the middle of the basic procedure by finding an interior solution to (P) or nonzero solution to (D);

\noindent (ii) Detecting that there is no $\varepsilon$-interior feasible solution.

\medskip
%%TT161130
Let $p_1 < \ldots < p_q$ be the iteration number in which the $k$-th block is transformed.
Let $G_k^{p_1}, \ldots, G_k^{p_q}$ be the matrix of the automorphism
transformation associated with ${\cal K}_k$ performed in the course of the algorithm.  Then it follows that
\[
{\cal F}_k \subseteq G_k^{(p_{q})} G_{k}^{(p_1)} \ldots G_k^{(p_1)} {\cal C}(e_k, e_k).
\]
So, in view of Proposition \ref{p3}, we conclude (ii) when the following relation holds:
\[
{\rm det} (G_k^{(p_q)}) \ldots {\rm det} (G_k^{(p_1)}) V_{d_k} \leq 0.96^{qd_k} V_{d_k} < \varepsilon^{d_k} V_{d_k}.
\]
This relation implies that $q$ is bounded by $O(\log \varepsilon^{-1})$.
%%TT161130 end
The most time consuming case is that the number of occurrence of cutting process is almost even 
$(\sim \log \varepsilon^{-1})$ for all cones
before termination of the algorithm.
Then, the number of execution of the basic procedure is bounded by $O(n \log \varepsilon^{-1})$.

Since $O(n\log\varepsilon^{-1})$ executions of the basic procedure might be necessary before completion of the whole procedure,
%%TT161223 the complexity of the proposed algorithm is $O(n (n^3{\bar n}+m^2{\bar n}+ mn{\bar n}) \log \varepsilon^{-1})$. 
the complexity of the proposed algorithm is $O(n (n^3{\bar n}\max_i d_i +m{\bar n}^2) \log \varepsilon^{-1})$. 

%%TTTTT\section{Main Algorithm: Full Version with Reduced Complexity}

%%TT161130
%The complexity of the basic version of the main algorithm introduced in the previous section is 
%%TT161223 $O(n (n^3{\bar n}{\max_i d_i}+m^2{\bar n}+ mn{\bar n}) \log \varepsilon^{-1})$ in terms of the number of arithmetic operations.  
%$O(n (n^3{\bar n}{\max_i d_i}+m{\bar n}^2) \log \varepsilon^{-1})$ in terms of the number of arithmetic operations.  
%In the context of linear programming, Chubanov developed a technique to reduce the complexity by making 
%makes use of an output of the previous basic procedure to initiate a new basic procedure.  We show that
%Chubanov's approach can be immediately generalized to our extension.
%%TT161130end
%The following procedure corresponds to Algorithm 2.1 of \cite{Ch13}.  
%The only difference which should be taken care in making 
%an analogous argument to Section 2.1 of \cite{Ch13} to reduce the complexity is that 
%the diagonal scaling matrix ``$D$"in \cite{Ch13}
%is replaced with a block diagonal matrix where some blocks are elements of $Aut({\cal K}_i)$ instead of the identity matrix.
%The following is the main algorithm where this reduction of complexity is implemented.  The algorithm is completely 
%synchronized with Algorithm 2.1 of \cite{Ch13}.

\section{Remarks}

Before concluding this paper, we make some remarks.

\subsection{Condition Number}

We define a condition number of (P) as follows:
\[
{\rm cond}(A,{\cal K}) = \min_{x\in {\cal F}}\left(\frac{\lambda_{\max}(x)}{\lambda_{\min}(x)} \right)
\]
If (P) have an interior feasible solution, ${\rm cond}(A,{\cal K})$ stays finite, but it becomes infinity if (P) is feasible but
is not interior feasible.  This quantity is useful in evaluating the complexity of the main algorithm developed
in this paper.  It is worth noting that ${\rm cond}(A,{\cal K}) = \varepsilon_P^{-1}$, where $\varepsilon_P$ is the optimal
value of the following problem:
\[
\max\ \varepsilon,\ \ \ A x = 0,\ \ \ e \succeq x \succeq \varepsilon e.
\]

\subsection{Running Time to Find a Feasible Solution to an Interior-feasible System}

Suppose that we set $\varepsilon=0$ and run the main algorithm.
The algorithm will never stop if (P) does not have an interior feasible solution.  But if there exists an 
interior feasible solution, then the algorithm is ensured to terminate in $O(n\log \varepsilon_P^{-1})$ execution of 
of the basic procedure, where $\varepsilon_P$ is the optimal value of the following optimization problem:
\[
\max \varepsilon\ \ \hbox{subject to}\ Ax = 0,\ \ \ \|x\|_\infty \leq 1,\ \ \ x\succeq \varepsilon e.
\]
In view of (\ref{equiv}), we have $(2\varepsilon_P)^{-1} \leq {\rm cond}(A,{\cal K}) \leq \varepsilon_P^{-1}$.  Therefore, the algorithm is 
capable of finding an interior solution to (P) in $O(n\log {\rm cond}(A,{\cal K}))$ times execution of the basic procedure.

\subsection{SOCP Feasibility Problem}

Suppose that we deal with the problem of finding an interior-feasible solution $x$ to 
\begin{equation}\label{aaaaa}
Ax =b,\ x\in {\widetilde {\cal K}},
\end{equation}
where ${\widetilde {\cal K}}$ is a direct product of $n$ second-order cones/half-lines. 
We assume the system is interior-feasible.
To solve this problem, we consider the homogenized system
\[
Ax -b\tau = 0,\ x \in {\widetilde {\cal K}},\ \tau \in {\RR}_+,
\]
where $\RR_+$ is a half-line.
We run the main algorithm with $\varepsilon = 0$. The algorithm stops in 
$O(n\log({\rm cond}((A\ -b);{\widetilde {\cal K}}\times {\RR}_+)))$ iterations. 

For any feasible solution to $\tilde x$, 
\[
\frac{\lambda_{\max}((\tilde x; 1))}{\lambda_{\min}((\tilde x; 1))} 
\]
is an upper bound for ${\rm cond}((A\ -b);{\widetilde{\cal K}}\times {\RR}_+)$.

This implies that if the system (\ref{aaaaa}) have an interior feasible 
solution whose components are more or less in the same magnitude, 
then less number of iterations is required to find an interior feasible solution.

\section{Conclusion}

We extended Chubanov's algorithm for linear programming to second-order cone programming.
The extension to semidefinite programming and symmetric cone programming is an interesting
topic for further research.
%%TTTTT
In the case of linear program, Chubanov \cite{Ch13} developed a technique to reduce the complexity by a factor of $n$ by
initiating a basic procedure using the second last iterate of the preceding basic procedure.
This idea does not directly carry over to the second-order cone program.  Extending the technique to the second-order cone
program is another interesting subject. 
%%TT161130
As was mentioned in introduction, Pe\~na and Soheili developed an extension
to Chubanov's algorithm to symmetric cone programming.  We hope that 
comparison of our extension and their algorithm will shed new insight into substance of Chubanov's idea in conic programming.
%%TT161130end

\section*{Acknowledgement}

%%TT161130
We would like to thank Dr.~Bruno F.~Louren\c{c}o of Seikei University for his careful reading of the first manuscript and suggesting improvement,
and for bringing the reference \cite{Pe16} to our attention. 
%%TT161130end
The aurhors are supported in part with Grant-in-Aid for Scientific Research (B), 2015, 15H02968,
from the Japan Society for the Promotion of Sciences.  The first author is supported in part with Grant-in-Aid for Young Scientists (B), 15K15941.

\section*{Appendix}

In this section, we describe how we can solve a standard SOCP problem with the algorithm developed in this paper.
Consider the pair of primal and dual SOCP:
\[
{\rm (P)}\ \min c^T x\ \ \ \hbox{subject to}\ Ax = b,\ \ \ x\succeq 0
\]
and
\[
{\rm (D)}\ \max b^T y\ \ \ \hbox{subject to}\ s=c- A^T y,\ \ \ s\succeq 0.
\]
Suppose (P) and (D) have interior feasible solutions.  Then
(P) and (D) have optimal solutions with the same optimal value.
Furthermore, the optimal set is bounded for the both problems.
In this appendix, 
we explain, given any $\delta > 0$, how the algorithm developed in this paper can be used to 
find a feasible solutions satisfying $c^T x - b^T y \leq \delta$.  If $\delta$ is sufficiently
small, then $x$, $y$ and $s$ are approximate optimal solutions to (P) and (D).

It is well-known that (P) and (D) are equivalent to the following problem.
\[
{\rm (PD)}\ {\rm find}\ Ax=b,\ \ \ c-A^Ty = s,\ \ \ c^T x - b^T y = 0,\ \ \ x \succeq 0,\ \ \ s\succeq 0.
\]

We have the following proposition.

\noindent{\bf Proposition A.1}
{\it 
Let $\tilde x$ and $(\tilde y, \tilde s)$ be an interior feasible solution to (P) and (D),
respectively.  Let 
\[
\hat\varepsilon = \min(\lambda_{\min}(\tilde x), \lambda_{\min}(\tilde s),1)\ \ \ \hbox{and}\ \ \ M = {c^T \tilde x - b^T \tilde y}.
\]
If $0\leq t \leq 1/2$, 
\[
({\rm PD}(t))\ \ \ {\rm find}\ Ax=b,\ \ \ c-A^Ty = s,\ \ \ 0 \leq c^T x - b^T y \leq 2 t M,\ \ \ x \succeq 0, \ \ \ s \succeq 0,
\]
has $t \min(\hat\varepsilon, M)$-interior feasible solution.
}

\begin{proof}
$(\tilde x, (\tilde y, \tilde s))$ is a $\hat\varepsilon$-interior feasible solution to (PD).
Let $x^*$, $(y^*; s^*)$ be optimal solutions to (P) and (D), and define
\[
x(t) = t\tilde x + (1-t) x^*,\ \ y(t) =  t\tilde y + (1-t) y^*,\ \ s(t)= t\tilde s + (1-t) s^*.
\]
Then, $(x(t), (y(t), s(t)))$ is a $t\hat \varepsilon$-interior feasible solution to (PD) for any $t\in [0,1]$.
We also have
\[
c^T x(t) - b^T y(t) = t (c^T \tilde x - b^T \tilde y)= tM.
\]
It is easy to check that $x(t)$, $(y(t),s(t))$ is indeed $t \min(\hat\varepsilon, M)$-interior feasible solution
to (PD($t$)) for $t\in [0,1/2]$, and we are done.

\end{proof}

We may consider $\tilde x$, $\tilde s$ and $\tilde y$ as a feasible solution obtained in Phase I.
Now we are ready to describe an algorithm to solve (P) and (D) approximately.
The algorithm works in two phases.

\begin{enumerate}

\item {\bf(Phase I)}\ We apply the feasibility algorithm described in Section 6.3
to 
\[
{\rm (PD)}\ 
Ax = b,\ s= c - A^T y,\ x \succeq0, \ s \succeq 0
\]
with $\varepsilon = 0$.
(This problem contains $y$ as a free variable, but
we can apply the main algorithm after rewriting the condition $s=c - A^T y$ with $P_A(c-s)=0$ to eliminate
$y$. In the end of the algorithm, we can recover $y$ from $s$.)
Then, we will find an interior feasible solution 
$(x, s, y) = (\tilde x, \tilde s, {\tilde y})$.

The complexity of this step is estimated with the result in Section 6.3,
in terms of the condition number. Let 
$\tilde\varepsilon = \min(\lambda_{\min}(\tilde x), \lambda_{\min}(\tilde s), 1)$.
Then, $(\tilde x, {\tilde s}, {\tilde y}))$ is an $\tilde\varepsilon$-interior 
feasible solution to (PD).  

\item {\bf (Phase II)}\ If we want to reduce the objective value by a factor of $t (\leq1/2)$ from $c^T \tilde x - b^T \tilde y$,
we solve the interior-feasibility problem (PD($t$)) above, which is ensured to have an 
$t\min(\tilde\varepsilon, c^T\tilde x - b^T\tilde y)$-interior feasible solution.
The complexity is estimated with the result in Section 6.3, again, in terms of the condition number of (PD($t$)).

\end{enumerate}


\thebibliography{99}
\bibitem{AlGo} F. Alizadeh and D. Goldfarb: Second-order cone programming.
{\it Mathematical programming}, Vol.95 (2003), pp.3-51.
\bibitem{BaDeJu}
A. Basu, J. A. De Loera, M. Junod: On Chubanov's method for linear programming.
{\it Informs journal on computing}, Vol.26 (2013), pp.336-350.
\bibitem{Ch10}
S. Chubanov: A polynomial relaxation-type algorithm for linear programming.
Optimization Online, February, 2011.
\bibitem{Ch12}
S. Chubanov: A strongly polynomial algorithm for linear systems having a binary
solution. {\it Mathematical Programming}, Vol.134 (2012), pp.533-570.
\bibitem{Ch13}
S. Chubanov: A polynomial projection algorithm for linear programming for linear feasibility problems.
{\it Mathematical Programming}, Vol.153 (2015), pp.687-713.
\bibitem{Cu15} F. Cucker, J. Pe\~{n}a and V. Roshchina: Solving second-order conic systems with variable precision.
{\it Mathematical Programming}, Vol.150 (2015), pp 217-250.
\bibitem{Kar} N. Karmarkar:
A New Polynomial Time Algorithm for Linear Programming. {\it Combinatorica}, Vol.4 (1984), pp.373-395.
\bibitem{Kha1}L. G. Khachiyan: A Polynomial Algorithm in Linear Programming. Doklady
Akademiia Nauk SSSR, Vol.244 (1979), pp.1093-1096 (in Russian, translated in Soviet Mathematics
Doklady, Vol.20 (1979), pp.191-194).
\bibitem{Kha2}L. G. Khachiyan:
Polynomial Algorithms in Linear Programming. Zhurnal
Vychisditel'noi Matematiki i Matematicheskoi Fiziki, Vol.20 (1980), pp.51-68 (in Russian). 
\bibitem{LiRoTe}
D. Li, K. Roos, and T. Terlaky:
A polynomial column-wise rescaling von Neumann algorithm. Optimization Onine, June, 2015.
\bibitem{Lo98} M. S. Lobo, , L. Vandenberghe, S. Boyd, H. Lebret:
Applications of second-order cone programming.
{\it Linear Algebra and its Applications}, Vol.284, (1998), pp.193-228.
\bibitem{MoTs} R. D. C. Monteiro and T. Tsuchiya: Polynomial convergence of primal-dual algorithms for the second-order cone program based on the MZ-family of directions.
{\it Mathematical Programming}, Vol.88 (2000), pp.61-83.
\bibitem{NeNe}
Y. E. Nesterov and A. S. Nemirovskii. {\it Interior point polynomial methods in convex programming:
Theory and Applications}, SIAM, Philadelphia, 1994.
\bibitem{NeTo}Y. E. Nesterov and M. J. Todd: Self-scaled barriers and interior-point methods for convex programming
{\it Mathematics of Operations research}, Vol.22, pp.1-42 (1997).

%%TT161130
%\bibitem{Pe16early} J. Pe\~na and N. Soheili: A deterministic rescaled perceptron algorithm. {\it Mathematical Programming}, Vol.155 (2016), pp 497-510.
\bibitem{Pe16} J. Pe\~na and N. Soheili: Solving conic systems via projection and rescaling. arXiv:1512.06154v2[math.OC], April, 2016.
%\bibitem{Ro14} K. Roos: On Chubanovfs algorithm for the linear feasibility problem.
%Slide presented at Mathematical Programming Seminar CORE, Louvain-la-Neuve December 16, 2014
%(available from https://www.uclouvain.be/cps/ucl/doc/core/documents/CORE\_Opt\_Sem\_final.pdf)
%%TT161130end
\bibitem{Ro15} K. Roos: An improved version of Chubanovfs method for solving a homogeneous feasibility problem.
Optimization online, November 2015 (Revised: July 2016).
\bibitem{Tsu}T. Tsuchiya: A Convergence Analysis of the Scaling-invariant Primal-dual Path-following Algorithms for Second-order Cone Programming. Optimization Methods and Software, Vol.11/12 (1999), pp.141-182.
\end{document}